\documentclass{amsart}

\usepackage{amssymb}

\makeatletter
\@namedef{subjclassname@2020}{%
  \textup{2020} Mathematics Subject Classification}
\makeatother

\newtheorem{theorem}{Theorem}[section]

\newtheorem{lemma}[theorem]{Lemma}

\theoremstyle{definition}

\begin{document}

\title[Multiplicative functions additive on cubes]{On multiplicative functions\\ which are additive on positive cubes}

\author{Poo-Sung Park}

\address{Department of Mathematics Education,
	Kyungnam University,
	Changwon, 51767,
	Republic of Korea
}
\email{pspark@kyungnam.ac.kr}

\thanks{This research was supported by Basic Science Research Program through the National Research Foundation of Korea (NRF) funded by the Ministry of Education (NRF-2021R1A2C1092930).}

\date{}

\begin{abstract}
Let $k \geq 3$. If a multiplicative function $f$ satisfies
\[
	f(a_1^3 + a_2^3 + \cdots + a_k^3) = f(a_1^3) + f(a_2^3) + \cdots + f(a_k^3)
\]
for all $a_1, a_2, \ldots, a_k \in \mathbb{N}$, then $f$ is the identity function. The set of positive cubes is said to be a $k$-additive uniqueness set for multiplicative functions. But, the condition for $k=2$ can be satisfied by infinitely many multiplicative functions.

Besides, if $k \geq 3$ and a multiplicative function $g$ satisfies
\[
	g(a_1^3 + a_2^3 + \cdots + a_k^3) = g(a_1)^3 + g(a_2)^3 + \cdots + g(a_k)^3
\]
for all $a_1, a_2, \ldots, a_k \in \mathbb{N}$, then $g$ is the identity function. However, when $k=2$, there exist three different types of multiplicative functions.
\end{abstract}



\maketitle

\section{Introduction}
The additive uniqueness, briefly AU, was initiated by Spiro in 1992 \cite{Spiro}. She call a set $E \subset \mathbb{N}$ a AU set for a set $S$ of arithmetic functions if the condition
\[
	f \in S \text{ and } f(a+b) = f(a)+f(b)\text{ for all } a,b \in E
\] 
determines $f$ uniquely. She showed that the set of prime numbers is AU for multiplicative functions $f$ with $f(p_0) \ne 0$ for some $p_0$ prime.

Two decades later, Fang \cite{Fang} and Dubickas and \v{S}arka \cite{DS} showed that the set of prime numbers with condition
\[
	f(p_1 + p_2 + \cdots + p_k) = f(p_1) + f(p_2) + \cdots + f(p_k)\text{ for fixed }k \geq 3
\]
also determines the multiplicative function $f$ uniquely if $f(p_0) \ne 0$ for some $p_0$ prime. Henceforth, let us call such a set $k$-AU.

In 1996 Chung  \cite{Chung1996} studied condition
\[
	f(m^2 + n^2) = f(m^2) + f(n^2)\text{ for all }m,n \in \mathbb{N}
\]
and showed that the set of positive squares is not AU for multiplicative functions.

The author, two decades later, showed that the set of positive squares is $k$-AU for $k \geq 3$ \cite{Park2018AequatMath}. This is the first case for $k$-AU but not $\ell$-AU with $\ell < k$.

In this paper we show that the set of positive cubes is a $k$-additive uniqueness set, like the set of positive squares, for multiplicative functions when $k \geq 3$.

\begin{theorem}\label{thm:main}
The set of positive cubes is $k$-AU for $k \geq 3$. That is, if a multiplicative function $f_k$ satisfies the condition
\[
	f_k(a_1^3 + a_2^3 + \cdots + a_k^3)
	= f_k(a_1^3) + f_k(a_2^3) + \cdots + f_k(a_k^3)
\]
for all $a_1, a_2, \ldots, a_k \in \mathbb{N}$, then $f_k$ is the identity function.
\end{theorem}

But, $f_2$ is not determined uniquely. So, we proceed as follows:
\begin{itemize}
\item[\S2] Classify the multiplicative functions $f_2$.
\item[\S3] Show that $f_3$ is the identity function.
\item[\S4] Show that $f_4$ is the identity function.
\item[\S5] Show that $f_k$ is the identity function with $k \geq 5$.
\end{itemize}

In \S6 we consider a variant problem: For fixed $k \geq 2$, if a multiplicative function $g_k$ satisfies
\[
	g_k(a_1^3 + a_2^3 + \cdots + a_k^3) = g_k(a_1)^3 + g_k(a_2)^3 + \cdots + g_k(a_k)^3
\]
for all $a_1, a_2, \ldots, a_k \in \mathbb{N}$, then is $g_k$ determined uniquely? If we set $K(x_1, \ldots, x_k) = x_1^3 + \cdots + x_k^3$, this problem is to find a multiplicative function $g_k$ commutable with $K$.

In 2014 Ba\v{s}i\'c \cite{Basic} showed that various multiplicative function $g$ satisfies
\[
	g(a^2 + b^2) = g(a)^2 + g(b)^2
\]
and the author, a couple of years later, showed that a multiplicative function $g$ is the identity function if $k \geq 3$ is fixed and 
\[
	g(a_1^2 + a_2 + \cdots + a_k^2) = g(a_1)^2 + g(a_2)^2 + \cdots + g(a_k)^2
\]
holds for all  $a_1, a_2, \ldots, a_k \in \mathbb{N}$ \cite{Park2018IJNT}.

In this section, we show that $g_k$ is uniquely determined for $k \geq 3$ and $g_2$ is one of three functions.

\section{non $2$-additive uniqueness}

The set of positive cubes is not (2-)AU for multiplicative functions. That is, although a multiplicative function $f_2$ satisfies
\[
	f_2(a^3+b^3) = f_2(a^3) + f_2(b^3)
\]
for all $a,b \in \mathbb{N}$, $f_2$ may not be the identity function.

Since $n^3 \equiv 0, 1, -1 \pmod{9}$, the sum of two cubes cannot be congruent to $\pm3$ modulo $9$. If we want to determine $f_2(3)$, we should use $f_2(3m)$ with $3 \nmid m$. But, since $3m \equiv \pm3 \pmod{9}$, $3m$ cannot be written as a sum of two positive cubes. So, for example, if we define a multiplicative function $f: \mathbb{N} \to \mathbb{C}$ by
\[
	f(n) =
	\begin{cases}
		0 & \text{ if }3 \mid n\text{ and }3^2 \nmid n, \\
		n & \text{ otherwise,}
	\end{cases}
\]
this function satisfies the additive condition but is not the identity function.

In general, we can characterize the function $f_2$ completely.

\begin{theorem}\label{thm:f2}
If a multiplicative function $f_2$ satisfies
\[
	f_2(a^3+b^3) = f_2(a^3) + f_2(b^3)
\]
for $a, b \in \mathbb{N}$, then $f_2(n) = n$ unless $3\,|\,n$ and $3^2 \nmid n$. We can set the value $f_2(3)$ arbitrarily. 
\end{theorem}

We will prove through the following process:
\begin{enumerate}
\item Determine $f_2(n)$ for some $n$'s.
\item Show that $f_2$ fixes $2^r$, $3^r$, and $7^r$ except for $3$.
\item Assume that $f_2$ fixes $q^r$ for all primes $q < p$ except for $f_2(3)$ to use induction on prime $p$.
\item Show that $f_2$ fixes $p$, $p^2-p+1$, $p^3$, $p^2+3$, $(p^2+1)^3$, $p^4+3$.
\item Show that $f_2(p^r) = p^r$.
\end{enumerate}

\begin{lemma}\label{lem:f2(n)=n_for_some_n}
$f_2(n) = n$ for $n = 1, 2, 2^2, 2^3, 3^2, 3^3, 5, 5^3, 7, 7^3, 13, 19$.
\end{lemma}

\begin{proof}
The following \texttt{Mathematica} code yields the required results except for $n=4$, where $x_n = f_2(n)$.
\begin{verbatim}
Solve[x1 == 1 && x2 == x1 + x1 && x9 == x8 + x1
 && x5 x7 == x27 + x8 && x5 x13 == x64 + x1 
 && x8 x9 == x64 + x8 && x7 x13 == x64 + x27 
 && x2 x9 x7 == x125 + x1 && x8 x19 == x125 + x27
 && x27 x7 == x125 + x64 && x8 x43 == x343 + x1 
 && x27 x13 == x343 + x8 && x27 x19 == x512 + x1 
 && x13 x43 == x343 + x8 x27 && x9 x5 x19 == x512 + x343,
{x1, x2, x5, x7, x8, x9, x13, x19, x27, x43, x64, x125, x343, x512}]
\end{verbatim} 

Then, $f_2(4) = 4$ can be derived from $f_2(4)\,f_2(7) = f_2(1^3+3^3) = 1+f_2(27)$.
\end{proof}

\begin{lemma}\label{lem:f2(2^r)}
$f_2(p^r) = p^r$ for $p \in \{2,3,7\}$ and $r \geq 1$ except for $f_2(3)$. 
\end{lemma}

\begin{proof}
We use induction on $r$. Assume that $2^r$ is fixed by $f_2$ for all $r \le 3m$. It is clear that $f_2(2^{3m+1}) = 2^{3m+1}$ by
\[
	f_2(2^{3m+1}) = f(2^{3m}+2^{3m}) = f(2^{3m}) + f(2^{3m}).
\]

We obtain $f_2(2^{3m+2}) = 2^{3m+2}$ by calculating $f_2( 7 \cdot 2^{3m+2} )$ in two ways:
\[
		f_2(7)\,f_2(2^{3m+2}) = 7 f_2(2^{3m+2}) \\
\]
and
\begin{align*}
	f_2( 7 \cdot 2^2 \cdot 2^{3m}) 
	&= f_2\bigl( (1^3+3^3)\,2^{3m} \bigr) 
		= f_2(2^{3m}) + f_2(3^3)\,f_2(2^{3m}) \\
	&= 2^{3m} + 3^3 \cdot 2^{3m}
		= 7 \cdot 2^{3m+2}.
\end{align*}

We can also obtain that $f_2(2^{3(m+1)}) = 2^{3(m+1)}$ by calculating $f_2( 9 \cdot 2^{3m})$ in two ways:
\[
		f_2(9)\,f_2(2^{3m}) = 9 \cdot 2^{3m}
\]
and
\[
	f_2\bigl( (1^3+2^3)\,2^{3m} \bigr) 
	= f_2(2^{3m}) + f_2(2^{3(m+1)})
	= 2^{3m} + f_2(2^{3(m+1)}).
\]

Assume that $3^r$ is fixed by $f_2$ for $2 \le r \le 3m$. Note that $f_2(11^3) = 11^3$ by
\[
	f_2(11^3+5^3) = f_2(11^3) + f_2(5^3) = f(2^4)\,f(7)\,f(13).
\]

Now, $f_2(3^{3m+1}) = 3^{3m+1}$ from
\begin{align*}
	f_2(7 \cdot 19 \cdot 3^{3m+1})
	&= f_2( 7 \cdot 19 \cdot 3^4 \cdot 3^{3(m-1)}) = f_2\bigl( (5^3+22^3)\,3^{3(m-1)} \bigr).
\end{align*}

It is easy to show that $f_2(3^{3m+2}) = 3^{3m+2}$ by
\begin{align*}
	f_2(3^{3m+2})
	&= f_2( 9 \cdot 3^{3m}) = f_2\bigl( (1^3+2^3)\,3^{3m} \bigr).
\end{align*}

We can deduce that $f_2(3^{3(m+1)}) = 3^{3(m+1)}$ by calculating
\begin{align*}
	f_2( 7 \cdot 3^{3(m+1)})
	&= f_2( 7 \cdot 3^3 \cdot 3^{3m}) = f_2\bigl( (4^3+5^3)\,3^{3m} \bigr).
\end{align*}

Now, we can show that $f_2$ fixes $7^{3m+1}$, $7^{3m+2}$, and $7^{3(m+1)}$ with $m \geq 0$ by calculating
\begin{align*}
	f_2(2^2 \cdot 7^{3m+1})
	&= f_2(2^2 \cdot 7 \cdot 7^{3m}) = f_2\bigl( (1^3+3^3)\,7^{3m} \bigr), \\
	f_2(11 \cdot 7^{3m+2})
	&= f_2(11 \cdot 7^2 \cdot 7^{3m}) = f_2\bigl( (3^3+8^3)\,7^{3m} \bigr), \\
	f_2(3^3 \cdot 13 \cdot 7^{3m})
	&= f_2\bigl( (2^3+7^3) 7^{3m} \bigr).
\end{align*}
This fact about $7^r$ will be used in Lemma~\ref{lem:f2_fixes_some_patterns}.
\end{proof}

Now, we use induction on prime $p \geq 5$ under the assumption
\[
	f_2(q^r) = q^r\text{ for all primes }q < p\text{ except for }f_2(3). \tag{$\ast$}\label{assumption}
\]

\begin{lemma}\label{lem:f2(p)}
Under the assumption (\ref{assumption}), $f_2(p) = p$ for all prime $p$.
\end{lemma}

\begin{proof}
Let $p = 2r+1 \geq 5$. Note that $f_2$ fixes $(r+1)^3$ and $r+2$, since their all prime divisors are less than $p$ and $3 \nmid (r+2)$. Now we can say that
\begin{align*}
	f_2\bigl( (r+1)^3+1^3 \bigr)
	&= f_2\bigl( (r+1)^3 \bigr) + f_2(1^3)
	= (r+1)^3 + 1 \\
	&= f_2\!\bigl( (r+2)(r^2+r+1) \bigr)
	= f_2(r+2)\,f_2(r^2+r+1) \\
	&= (r+2)\,f_2(r^2+r+1),
\end{align*}
since $\gcd(r+2, r^2+r+1) = \gcd(r+2, 3) = 1$. Hence, $f_2$ fixes $r^2+r+1$.

On the other hand, 
\begin{align*}
	f_2\bigl( (r+1)^3+r^3 \bigr)
	&= f_2\bigl( (r+1)^3+r^3 \bigr) + f_2(r^3) = (r+1)^3+r^3 \\
	&= f_2\bigl( (2r+1)(r^2+r+1) \bigr)
	= f_2(2r+1)\,f_2(r^2+r+1) \\
	&= f_2(p)\,(r^2+r+1)
\end{align*}
holds, since all prime divisors of $r^3$ is less than $p$ and $\gcd(2r+1, r^2+r+1) = 1$. Hence, $f_2(p) = p$.
\end{proof}

\begin{lemma}\label{lem:f2(p^3),f2(p^2-p+1)}
Under the assumption (\ref{assumption}), $f_2$ fixes $p^3$ and $p^2-p+1$ or $\frac{p^2-p+1}{3}$.
\end{lemma}

\begin{proof}
Let $p = 3^r s-1$ with $r \geq 1$ and $3 \nmid s$. Note that
\begin{align*}
	f_2( p^3 + 1^3)
	&= f_2(3^{r+1})\,f_2\!\left( \frac{p+1}{3^r} \right) f_2\!\left( \frac{p^2-p+1}{3} \right), \\
	f_2\bigl(p^3 + (p-1)^3 \bigr)
	&= f_2(3^2)\,f_2\!\left( \frac{2p-1}{3} \right) f_2\!\left( \frac{p^2-p+1}{3} \right),
\end{align*}
since $\gcd(p+1, p^2-p+1) = \gcd(2p-1, p^2-p+1) = 3$ and $\frac{p+1}{3^r}$, $\frac{p^2-p+1}{3}$, $\frac{2p-1}{3}$ are indivisible by $3$.

By induction hypothesis, we can say $(p-1)^3$, $\frac{p+1}{3^r}$ and $\frac{2p-1}{3}$ are fixed by $f_2$ since their prime divisors are less than $p$. Then, we can deduce that
\[
	f_2(p^3) = p^3 
	\qquad\text{ and }\qquad
	f_2\!\left(\frac{p^2-p+1}{3}\right) = \frac{p^2-p+1}{3}.
\]

If $p = 3^r s+1$ with $r \geq 1$ and $3 \nmid s$, then consider
\begin{align*}
	f_2\bigl( (p+1)^3 + 1^3 \bigr)
	&= f_2(3^2)\,f_2\!\left( \frac{2p+1}{3} \right) f_2\!\left( \frac{p^2+p+1}{3} \right), \\
	f_2\bigl( (p+1)^3 + p^3 \bigr)
	&= f_2(3^2)\,f_2\!\left( \frac{p+2}{3} \right) f_2\!\left( \frac{p^2+p+1}{3} \right),
\end{align*}
where $\gcd(2p+1, p^2+p+1) = \gcd(p+2, p^2+p+1) = 3$ and $\frac{2p+1}{3}$, $\frac{p^2+p+1}{3}$, $\frac{p+2}{3}$ are indivisible by $3$.

Since $p+1$ is even and $\frac{p+1}{2} < p$, we can say that $(p+1)^3$ is fixed by $f_2$. Then, 
\[
	f_2(p^3) = p^3 
	\qquad\text{ and }\qquad
	f_2\!\left( \frac{p^2+p+1}{3} \right) = \frac{p^2+p+1}{3}
\]
by induction hypothesis. This yields $f_2(p^2-p+1) = p^2-p+1$, since $\gcd(p+1, p^2-p+1) = 1$ and
\begin{align*}
	f_2( p^3 + 1^3)
	&= f_2( p+1 )\,f_2( p^2-p+1 ).
\end{align*}

We can conclude that $f_2$ fixes $\frac{p^2-p+1}{3}$ if $p \equiv -1 \pmod{3}$ and fixes $p^2-p+1$ otherwise.
\end{proof}

\begin{lemma}\label{lem:f2_fixes_some_patterns}
Under the assumption (\ref{assumption}), $f_2$ fixes $p^2+3$, $(p^2+1)^3$, and $p^4+3$.
\end{lemma}

\begin{proof}
Note that $\left(\frac{p+1}{2}\right)^3$ and $\left(\frac{p-1}{2}\right)^3$ are fixed by $f_2$ since $\frac{p+1}{2}$ and $\frac{p-1}{2}$ are less than $p$. Thus, we can conclude that $f_2(p^2+3) = p^2+3$ from
\[
	f_2\!\left( \left(\frac{p+1}{2}\right)^3 + \left(\frac{p-1}{2}\right)^3 \right) 
	= f_2(p)\,f_2\!\left( \frac{p^2+3}{4} \right)
	= p\,f_2\!\left( \frac{p^2+3}{4} \right).
\]

Next, consider
\[
	\left( \frac{p^2+1}{2} \right)^3 + \left( \frac{p-1}{2} \right)^3 
	= p \cdot \frac{p+1}{2} \cdot (p^2-p+1) \cdot \frac{p^2+3}{4},
\]
where
\begin{align*}
	&\gcd\!\left(p, \frac{p+1}{2} \right) = \gcd\left(p, p^2-p+1 \right) = \gcd\!\left(p, \frac{p^2+3}{4} \right) = 1,\\
	&\gcd\!\left(\frac{p+1}{2}, p^2-p+1 \right) =
	\begin{cases}
		1 & \text{if } p \equiv 1 \pmod{3},\\
		3 & \text{if } p \equiv -1 \pmod{3},
	\end{cases}\\
	&\gcd\!\left( \frac{p+1}{2}, \frac{p^2+3}{4} \right) = 1, \\
	&\gcd\!\left( p^2-p+1, \frac{p^2+3}{4} \right) =
	\begin{cases}
		1 & \text{if } p \not\equiv 5 \pmod{7}, \\
		7 & \text{if } p \equiv 5 \pmod{7}.
	\end{cases}
\end{align*}
We can verify that $(p^2+1)^3$ is fixed by $f_2$ for each case.

For example, assume $p = 3^r s - 1$ with $3 \nmid s$ and $p \equiv 5 \pmod{7}$. Then, since $3 \mid (p^2-p+1)$ but $3^2 \nmid (p^2-p+1)$,
\begin{align*}
	&f_2\!\left( \left( \frac{p^2+1}{2} \right)^3 + \left( \frac{p-1}{2} \right)^3 \right) 
		= f_2\!\left( \left( \frac{p^2+1}{2} \right)^3 \right)
			+ \left( \frac{p-1}{2} \right)^3 \\
	&= f_2(3^{r+1}) \, f_2(7^{a+b}) \, f_2(p) \, f_2\!\left( \frac{p+1}{2 \cdot 3^r} \right) \, f_2\!\left( \frac{p^2-p+1}{3 \cdot 7^a} \right) \, f_2\!\left( \frac{p^2+3}{4 \cdot 7^b} \right)
\end{align*}
with $a, b \geq 1$ and all terms on RHS are fixed by $f_2$. Thus, $f_2\bigl( (p^2+1)^3 \bigr) =  (p^2+1)^3$. 

Finally, note that
\begin{align*}
	f_2\!\left(
		\left( \frac{p^2+1}{2} \right)^3 + 1^3
	\right)
	&= 	
	\left(
		\frac{p^2+1}{2} 
	\right)^3
	+ 1^3.
\end{align*}

On the other hand,
\begin{align*}
	f_2\!\left(
		\left( \frac{p^2+1}{2} \right)^3 + 1^3
	\right)
	&= f_2(2)\,
	f_2\!\left( \frac{p^2+3}{4} \right)
	f_2\!\left( \frac{p^4+3}{4} \right)
\end{align*}
since $\gcd\left( \frac{p^2+3}{4}, \frac{p^4+3}{4} \right) = 1$.

Then, we can say that $f_2$ fixes $p^4+3$. 
\end{proof}

Now we are ready to prove Theorem \ref{thm:f2}. It is enough to show that $f_2(p^r) = p^r$ for all prime $p \geq 5$ and $r \in \mathbb{N}$. We use induction on $p$.

It is proved in Lemmas~\ref{lem:f2(p)} and \ref{lem:f2(p^3),f2(p^2-p+1)} that $f_2(p) = p$ and $f_2(p^3) = p^3$. Suppose that $f_2(p^{3m}) = p^{3m}$ for some $m \geq 1$. To show $f_2(p^{3m+1}) = p^{3m+1}$ we use the equality
\begin{align*}
	f_2\!\left( p^{3m}\left(\frac{p+1}{2}\right)^3 + p^{3m}\left(\frac{p-1}{2}\right)^3 \right) 
	= f_2(p^{3m+1})\,f_2\!\left( \frac{p^2+3}{4} \right),
\end{align*}
which was slightly modified from that used in Lemma~\ref{lem:f2_fixes_some_patterns}. Since $(p \pm 1)^3$ and $p^2+3$ are fixed, we obtain $f_2(p^{3m+1}) = p^{3m+1}$.

Similarly, $f_2(p^{3m+2}) = p^{3m+2}$ from
\begin{align*}
	f_2\!\left( p^{3m}\left(\frac{p^2-1}{2}\right)^3 + p^{3m}\left(\frac{p^2+1}{2}\right)^3 \right) 
	= f_2(p^{3m+2})\,f_2\!\left( \frac{p^4+3}{4} \right),
\end{align*}
where $\frac{p-1}{2} < \frac{p+1}{2} < p$ and $f_2$ fixes $(p^2+1)^3$ and $p^4+3$. This also works for $m=0$. Thus, $f_2(p^2) = p^2$.

It can be verified that $f_2(p^{3(m+1)}) = p^{3(m+1)}$ from
\begin{align*}
	f_2\bigl( p^{3m}(p^3+1) \bigr)
	&= f_2( p^{3(m+1)} + p^{3m} ) \\
	&= f_2( p^{3(m+1)}) + f_2(p^{3m}) 
	= f_2(p^{3(m+1)}) + p^{3m} \\
	&= f_2(p^{3m})\,f_2(p^3+1)
	= p^{3m} (p^3+1).
\end{align*}

Thus, proof is completed.

\section{$3$-additive uniqueness}

In the previous section, the set of positive cubes is not 2-AU. But, in this section, we show that the set of positive cubes is 3-AU. The key idea is algebraic representations of rational numbers as sums of three cubes of rational numbers.

Let $f_3$ be a multiplicative function satisfying the condition 
\[
	f_3(a^3+b^3+c^3) = f_3(a^3) + f_3(b^3) + f_3(c^3).
\]

\begin{lemma}\label{lem:f3(n)=n_for_some_n}
$f_3(n)=n$ for $n=1, 2, 2^3, 3, 3^2, 3^3, 5, 5^3, 7, 7^3, 11, 17, 23$.
\end{lemma}

\begin{proof}
It is clear that $f_3(3) = 3$.

If we set $y_n = f_3(n)$, then the following \texttt{Mathematica} code yields the required results.
\begin{verbatim}
Solve[y1==1 && y3==3
 && y2 y5 == y1 + y1 + y8 && y17 == y1 + y8 + y8 
 && y4 y9 = y8 + y8 + y27 && y43 = y8 + y27 + y27 
 && y5 y11 == y1 + y27 + y27 && y2 y31 == y8 + y27 + y27 
 && y2 y3 y11 == y1 + y1 + y64 && y4 y23 == y1 + y27 + y64 
 && y9 y11 == y8 + y27 + y64 && y3 y43 == y1 + y64 + y64 
 && y8 y17 == y8 + y64 + y64 && y9 y17 == y8 + y64 + y64 
 && y5 y31 == y27 + y64 + y64 && y8 y27 == y27 + y64 + y125 
 && y3 y5 y23 == y1 + y1 + y343 && y2 y27 y7 == y27 + y64 + y125,
{y1, y2, y3, y5, y7, y8, y9, y11, y17, y23, y27, y31, y43, y64,
 y125, y343}]
\end{verbatim}

\end{proof}

\begin{lemma}\label{lem:f3(2^r),f3(3^r)}
$f_3$ fixes $2^r$ and $3^r$ with $r \geq 1$.
\end{lemma}

\begin{proof}
We use induction. Assume that $f_2$ fixes $2^r$ for all $r < 3m$.

We can say that $f_3(2^{3m+1}) = 2^{3m+1}$ by
\begin{align*}
	f_3(23 \cdot 2^{3m+1})
	&= f_3(23)\,f_3(2^{3m+1}) = 23 f_3(2^{3m+1})\\
	&= f_3(23 \cdot 16 \cdot 2^{3(m-1)}) = f_3\bigl( (3^3+5^3+6^3)2^{3(m-1)} \bigr) \\
	&= f_3(3^3)\,f_3(2^{3(m-1)}) + f_3(5^3)\,f_3(2^{3(m-1)}) + f_3(3^3)\,f_3(2^{3m}).
\end{align*}

Next, $f_3(2^{3m+2}) = 2^{3m+2}$ by
\begin{align*}
	f_3(5 \cdot 2^{3m+2})
	&= f_3(5)\,f_3(2^{3m+2}) = 5 f_3(2^{3m+2})\\
	&= f_3(5 \cdot 2^5 \cdot 2^{3(m-1)}) = f_3\bigl( (2^3+3^3+5^3)2^{3(m-1)} \bigr) \\
	&= f_3(2^{3m}) + f_3(3^3)\,f_3(2^{3(m-1)}) + f_3(5^3)\,f_3(2^{3(m-1)}).
\end{align*}

Finally, $f_3(2^{3(m+1)}) = 2^{3(m+1)}$ from
\begin{align*}
	f_3(5 \cdot 2^{3m+1})
	&= f_3(5)\,f_3(2^{3m+1}) = 5 f_3(2^{3m+1})\\
	&= f_3(10 \cdot 2^{3m}) = f_3\bigl( (1^3+1^3+2^3)2^{3m} \bigr) \\
	&= f_3(2^{3m}) + f_3(2^{3m}) + f_3(2^{3(m+1)}).
\end{align*}

Now, assume that $f_3$ fixes upto $3^{3m}$. We can say that $f_3(3^{3m+1}) = 3^{3m+1}$ by
\begin{align*}
	f_3(3^{3m+1}) &= f_3(3^{3m}+3^{3m}+3^{3m}) = 3 f_3(3^{3m}).
\end{align*}

Let us consider $f_3(3^{3m+2})$. Note that $f_3$ fixes $2^6$ and $5^3$. Then, $f_3(3^{3m+2}) = 3^{3m+2}$ by
\begin{align*}
	f_3(2^4 \cdot 3^{3m+2}) 
	&= f_3(2^4)\,f_3(3^{3m+2}) = 2^4 \cdot f_3(3^{3m+2}) \\
	&= f_3(2^4 \cdot 3^5 \cdot 3^{3(m-1)}) = f_3\bigl( (1^3+8^3+15^3)\,3^{3(m-1)} \bigr) \\
	&= f_3(3^{3(m-1)}) + f_3(2^6)\,f_3(3^{3(m-1)}) + f_3(5^3)\,f_3(3^{3m}).
\end{align*}

Finally, $f_3(3^{3(m+1)}) = 3^{3(m+1)}$ by
\begin{align*}
	f_3(4 \cdot 3^{3m+2}) 
	&= f_3(4)\,f_3(3^{3m+2}) = 4 f_3(3^{3m+2}) \\
	&= f_3(4 \cdot 9 \cdot 3^{3m}) = f_3\bigl( (1^3+2^3+3^3)\,3^{3m} \bigr) \\
	&= f_3(3^{3m}) + f_3(2^3)\,f_3(3^{3m}) + f_3(3^{3(m+1)}).
\end{align*}%
\end{proof}

\begin{lemma}\label{lem:f3(n^3)=n^3}
$f_3$ fixes $n^3$.
\end{lemma}

\begin{proof}
It is enough to show that $f_3(p^{3r}) = p^{3r}$ for prime $p \geq 5$ and $r\geq 1$.
It was verified in Lemma~\ref{lem:f3(n)=n_for_some_n} that $f_3$ fixes $5^3$ and $7^3$.

Let $n = p^r = 6m \pm 1$ with $m \geq 2$. We can say $f_3(n^3) = n^3$ inductively by
\[
	m^3 + (5m)^3 + (5m\pm2)^3
	= (2m\mp1)^3 + (3m\pm2)^3 + (6m\pm1)^3.
\]
\end{proof}

Now, we prove Theorem \ref{thm:main} for $k=3$. By the above results, we may show only $f_3(p^r) = p^r$ for prime $p \geq 5$ and $r \in \mathbb{N}$.

In 1923 Richmond proved that every positive rational number, say $R$, can be written as a sum of positive cubes of rational numbers \cite{Richmond}. In especial, he showed 
\begin{align*}
	R \bigl( 6 \lambda^2 (\lambda^3 + 3 R)^2 \bigr)^3 
	&= \bigl( -(\lambda^3 + 3 R) (\lambda^6 - 30 \lambda^3 R + 9 R^2) \bigr)^3 \\
	&+ \bigl( \lambda^9 + 45 \lambda^6 R - 81 \lambda^3 R^2 + 27 R^3 \bigr)^3 \\
	&+ \bigl( 36 \lambda^3 (3R-\lambda^3) R \bigr)^3
\end{align*}
and each term is positive when
\[
	\frac{1-\sqrt{4/3}\sin 20^\circ}{1+\sqrt{4/3}\sin 20^\circ} \times 3R = .43376\dots \times 3R < \lambda^3 < 3R.
\]

We can check the representation by the \texttt{Mathematica} code:
\begin{verbatim}
Factor[( -(lambda^3 + 3 R) (lambda^6 - 30 lambda^3 R + 9 R^2) )^3
+ ( lambda^9 + 45 lambda^6 R - 81 lambda^3 R^2 + 27 R^3 )^3
+ ( 36 lambda^3 (3R-lambda^3) R )^3]
\end{verbatim}

Let $R$ be a prime $p > 23$. Then the length of interval 
\[
	(\sqrt[3]{.43376\dots \times 3p}, \sqrt[3]{3p})
\]
is greater than $1$. So the interval contains at least one integer $\lambda$. 

But, since there exist no integer $t$ such that
\[
	.43376\dots \times 3p = 1.30128\dots \times p < (pt)^3 < 3p,
\]
we can choose $\lambda$ relatively prime to $p$ in the interval. As for other primes $\le 23$ which are not proved to be fixed, we can choose a suitable $\lambda$ as follows:
\begin{align*}
	p=13&: \lambda = 3 \in (2.56707\dots, 3.39121\dots) \\
	p=19&: \lambda = 3 \in (2.91323\dots, 3.8485\dots)
\end{align*}
Then, $\gcd(p,6 \lambda^2 (\lambda^3 + 3 p)^2) = 1$ and
\[
	f_3\bigl(p \bigl( 6 \lambda^2 (\lambda^3 + 3 p)^2 \bigr)^3 \bigr)
	= f_3(p)\,\bigl( 6 \lambda^2 (\lambda^3 + 3 p)^2 \bigr)^3
\]
by Lemma~\ref{lem:f3(n^3)=n^3}. Thus, $f_3(p) = p$.

For example, if we let $p=29$ and choose $\lambda = 4$ from the interval $(3.3542, 4.43105)$, then
\[
	29 \cdot 2188896^3 = 6646265^3 + 1906183^3 + 1536768^3
\]
and thus $f_3(29) = 29$. 

If $R = p^2 \geq 25$, then the interval
\[
	(\sqrt[3]{.43376\dots \times 3p^2}, \sqrt[3]{3p^2})
\] 
contains an integer $\lambda$ relatively prime to $p^2$. So, $f_3(p^2) = p^2$ by the similar way.

If $R = p^3$, then $f_3(p^3) = p^3$ by Lemma~\ref{lem:f3(n^3)=n^3}. If $R = p^r$ with $p \geq 5$ and $r \geq 4$, then the length of the inteval
\[
	(\sqrt[3]{.43376\dots \times 3p^r}, \sqrt[3]{3p^r})
\] 
is greater than $2$. So we can find $\lambda$ relatively prime to $p^r$. Therefore, we can conclude that $f_3(p^r) = p^r$.

\section{$4$-additive uniqueness}

Let $f_4$ be a multiplicative function satisfying
\[
	f_4(a^3+b^3+c^3+d^4) = f_4(a^3) + f_4(b^3) + f_4(c^3) + f_4(d^3).
\]

Trivially, $f_4(1) = 1$ and $f_4(4) = 4$.

We show that $f_4(n) = n$ through the some steps.
\begin{enumerate}
\item $f_4(2^r) = 2^r$ for all positive integers $r$
\item $f_4(n^3) = n^3$ for all positive integers $n$
\item $f_4(p^r) = p^r$ for all primes $p$
\end{enumerate}

\begin{lemma}\label{lem:f4(2^r)}
$f_4(2^r) = 2^r$.
\end{lemma}

\begin{proof}
If we set $z_n = f_4(n)$, then the following \textsf{Mathematica} code yields
\[
	f_4(n) = n\text{ for }n = 1, 2, 2^2, 2^3, 3, 3^2, 3^3, \text{ and }5.
\]

\begin{verbatim}
Solve[z1 == 1 && z4 == 4 && z11 == z8 + z1 + z1 + z1
 && z2 z9 == z8 + z8 + z1 + z1 && z25 == z8 + z8 + z8 + z1
 && z2 z3 z5 == z27 + z1 + z1 + z1 && z37 == z27 + z8 + z1 + z1
 && z4 z11 == z27 + z8 + z8 + z1 && z8 z7 == z27 + z27 + z1 + z1
 && z9 z7 == z27 + z27 + z8 + z1 && z2 z5 z7 == z27 + z27 + z8 + z8
 && z2 z37 == z64 + z8 + z1 + z1 && z8 z11 == z64 + z8 + z8 + z8
 && z4 z25 == z64 + z27 + z8 + z1, 
 {z1, z2, z3, z4, z5, z7, z8, z9, z11, z25, z27, z37, z64}]  
\end{verbatim}

We use induction. We obtain $f_4(2^{3m+1}) = 2^{3m+1}$ is from
\begin{align*}
	f_4(3 \cdot 5 \cdot 2^{3m+1})
	&= 15 \cdot f_4(2^{3m+1}) \\ 
	&= f_4(30 \cdot 2^{3m}) = f\bigl( (1^3+1^3+1^3+3^3)\,2^{3m} \bigr) \\
	&= 30 \cdot f_4(2^{3m})
\end{align*}

Trivially, $f_4(2^{3m+2}) = 2^{3m+2}$ by
\begin{align*}
	f_4(2^{3m+2}) 
	&= f_4(2^{3m} + 2^{3m} + 2^{3m} + 2^{3m} ).
\end{align*}

Next, 
$f_4(2^{3(m+1)}) = 2^{3(m+1)}$ from
\begin{align*}
    f_4(9 \cdot 2^{3m+1}) 
    &= f_4(9)\,f_4(2^{3m+1}) = 9 \cdot 2^{3m+1} \\
    &= f_4(9 \cdot 2 \cdot 2^{3m}) = f\bigl( (1^3 + 1^3 + 2^3 + 2^3)\,2^{3m} \bigr) \\
    &= f_4(2^{3m}) + f_4(2^{3m}) + f_4(2^{3(m+1)}) + f_4(2^{3(m+1)}) \\
    &= 2 \cdot 2^{3m} + 2 \cdot f_4(2^{3(m+1)}).
\end{align*}

\end{proof}

\begin{lemma}\label{lem:f4(n^3)}
$f_4(n^3) = n^3$.
\end{lemma}

\begin{proof}
We have $f_4(2^r) = 2^r$ and $f_4(3^3) = 3^3$. For odd integer $n=2s+1$ with $s \geq 2$ we use induction and the following equality:
\begin{align*}
	&(2s + 1)^3 + (2s + 1)^3 + (s + 2)^3 + (s - 1)^3\\
	&= (s + 1)^3 + s^3 + \bigl(2(s + 1)\bigr)^3 + (2s)^3.
\end{align*}
\end{proof}

By the above Lemma, if $n$ can be written as a sum of four positive cubes, then $f_4(n) = n$.

Now, we are ready to prove $f_4(n)=n$ for arbitrary positive integer $n$. We need a deep result about sums of four positive cubes. It is a well-known conjecture that all sufficiently large positive natural numbers may be written as the sum of four positive cubes. But, it is known that
\[
	\frac{\#\mathcal{E}_4(N)}{N} \to 0 \quad\text{ as }N \to \infty,
\]
where $\mathcal{E}_4(N)$ is the set of numbers $\le N$ whose elements cannot be represented as sums of four positive cubes \cite{Davenport, Vaughan1986, Brudern1989, Brudern1991}.

Consider $f_4(p^r)$ for $p$ prime. Denote the set of positive integers $\le N$ by $[N]$ and the subset of $[N]$ whose elements are not divisible by $p$ by $[N]_p$. Then, every element of $[N]_p \setminus \mathcal{E}_4(N)$ can be written as a sum of four positive cubes.

Note that 
\begin{align*}
	\frac{\#\bigl( p^r ([N]_p \cap \Sigma^3_4 ) \bigr)}{\#[p^r N]}
	&= \frac{\#\bigl( p^r ([N]_p \setminus \mathcal{E}_4(N))\bigr)}{\#[p^r N]}\\
	&= \frac{\#\bigl( [N]_p \setminus \mathcal{E}_4(N)\bigr)}{p^r N}
	\to \frac{1}{p^r}\bigl( 1 - \frac{1}{p} \bigr) > 0
\end{align*}
as $N \to \infty$, where $\Sigma^3_4$ denotes the set of sums of $4$ positive cubes. Thus, there exists a positive integer $M$ such that $p \nmid M$, $M \in \Sigma^3_4$ and $p^r M \in \Sigma^3_4$.

Then, $f_4(p^r) = p^r$ since
\begin{align*}
	p^r M = f_4(p^r M) = f_4(p^r)\,f_4(M) = f_4(p^r) M.
\end{align*}
We are done.

\section{$k$-additive uniqueness}

Let $k \geq 5$. Suppose a multiplicative function $f_k$ satisfies
\[
	f_k(a_1^3 + a_2^3 + \dots + a_k^3) = f_k(a_1^3) + f_k(a_2^3) + \dots + f_k(a_k^3)
\]
for all positive integers $a_1, a_2, \dots, a_k$. 

\begin{lemma}\label{lem:fk(n^3)}
Either $f_k(n^3) = n^3$ for all $n \in \mathbb{N}$.
\end{lemma}

\begin{proof}
Note that
\begin{align*}
	f_k\bigl(1^3 + 3^3 \cdot 4^3 +(k-3)1^3 \bigr) &= f_k\bigl(9^3 + 2^3 \cdot 5^3 + (k-3)1^3 \bigr), \\
	f_k\bigl(1^3 + 5^3 + 5^3 + (k-3)1^3 \bigr) &= f_k\bigl(2^3 + 3^3 + 2^3 \cdot 3^3 + (k-3)1^3 \bigr), \\
	f_k\bigl(1^3 + 2^3 + 2^3 \cdot 5^3 + (k-3)1^3 \bigr) &= f_k\bigl(4^3 + 2^3 \cdot 3^3 + 9^3 + (k-3)1^3 \bigr), \\
	f_k\bigl(1^3 + 1^3 + 1^3 + 2^3 \cdot 3^3 + (k-4)1^3 \bigr) &= f_k\bigl(3^3 + 4^3 + 4^3 + 4^3 + (k-4)1^3 \bigr), \\
	f_k\bigl(1^3 + 1^3 + 2^3 + 8^3 + (k-4)1^3 \bigr) &= f_k\bigl(3^3 + 3^3 + 5^3 + 7^3 + (k-4)1^3 \bigr), \\
	f_k\bigl(1^3 + 4^3 + 5^3 + 8^3 + (k-4)1^3 \bigr) &= f_k\bigl(2^3 + 2^3 + 7^3 + 7^3 + (k-4)1^3 \bigr), \\
	f_k\bigl(1^3 + 1^3 + 3^3 + 9^3 + (k-4)1^3 \bigr) &= f_k\bigl(2^3 + 4^3 + 7^3 + 7^3 + (k-4)1^3 \bigr), \\
	f_k\bigl(1^3 + 1^3 + 3^3 + 4^3 + 4^3 + (k-5)1^3 \bigr) &= f_k\bigl(2^3 + 2^3 + 2^3 + 2^3 + 5^3 + (k-5)1^3 \bigr).
\end{align*}
Then, the following \texttt{Mathematica} code yields
\begin{align*}
	&t_n = f_k(n^3) = 1\quad\text{ for }n=2,3,4,5,7,8,9
\intertext{or}
	&t_n = f_k(n^3) = n^3\quad\text{ for }n=2,3,4,5,7,8,9.
\end{align*}

\begin{verbatim}
	Solve[1 + t3 t4 == t9 + t2 t5 && 1 + t5 + t5 == t2 + t3 + t2 t3 
	 && 1 + t2 + t2 t5 == t4 + t2 t3 + t9 && 3 + t2 t3 == t3 + 3 t4 
	 && 2 + t2 + t8 == 2 t3 + t5 + t7 && 1 + t4 + t5 + t8 == 2 t2 + 2 t7 
	 && 2 + t3 + t9 == t2 + t4 + 2 t7 && 2 + t3 + 2 t4 == 4 t2 + t5, 
	{t2, t3, t4, t5, t7, t8, t9}]
\end{verbatim}

Now, we can say that either $f_k(n^3) = 1$ or $f_k(n^3) = n^3$ for $n \le 10$. 

We use induction on two cases. Assume that $n \geq 11$ and $f_k(m) = 1$ for $m \le n$. Then, $f_k(n^3) = 1$ inductively by 
\begin{align*}
	&f_k\bigl( n^3 + (n-3)^3 + (n-3)^3 + \overbrace{2^3 + \cdots + 2^3}^{(k-3)\text{ summands}} \,\bigr) \\
	&= f_k\bigl( (n-1)^3 + (n-1)^3 + (n-4)^3 + \underbrace{2^3 + \cdots + 2^3}_{(k-5)\text{ summands}} + 1^3+3^3 \bigr).
\end{align*}

Note that $6^3 = 3^3 + 4^3 + 5^3$ and $7^3 = 1^3 + 1^3 + 5^3 + 6^3$. If $N^3 = a_1^3 + a_2^3 + \dots + a_r^3$ with $a_i \geq 1$, then
\begin{align*}
	(6N)^3 
	&= (6a_1)^3 + (6a_2)^3 + \cdots + (6a_{r-1})^3 + 6^3 a_r^3 \\
	&= (6a_1)^3 + (6a_2)^3 + \cdots + (6a_{r-1})^3 + (3a_r)^3 + (4a_r)^3 + (5a_r)^3  \\
	(7N)^3 
	&= (7a_1)^3 + (7a_2)^3 + \cdots + (7a_{r-1})^3 + 7^3 a_r^3 \\
	&= (7a_1)^3 + (7a_2)^3 + \cdots + (7a_{r-1})^3 + a_r^3 + a_r^3 + (5a_r)^3 + (6a_r)^3.
\end{align*}
Thus, multiplying $6^3$ or $7^3$ repeatedly, we can find a large cubic number $N^3$ which is a sum of $k$ positive cubes $a_1^3, a_2^3, \dots, a_k^3$ starting with any cubic number.

However, then, we obtain a contradiction:
\[
	1 = f_k(N^3) = f_k(a_1^3) + f_k(a_2^3) + \dots + f_k(a_k^3) = 1 + 1 + \dots + 1 = k. 
\]
So we can exclude the case $f(n^3) = 1$.

By the equality used in the previous case, we can obtain $f_k(n^3) = n^3$ for all $n$ inductively. 
\end{proof}

Now, we prove Theorem \ref{thm:main} for $k \geq 5$. We may consider $f_k(p^r)$. If $p^r$ can be written as a sum of $k$ positive cubes, then $f(p^r) = p^r$ trivially. But, if $p^r$ is not the sum of $k$ positive cubes, we need some tricks.

As mentioned in \S4, almost all positive integers can be written as sums of four positive cubes. If $n$ is the sum of four positive cubes, then $n+(k-4)$ is a sum of $k$ positive cubes. So, almost all positive numbers can be written as sums of $k$ positive cubes. In particular, all positive integers but finitely many can be written as sums of positive $k$ cubes when $k \geq 9$, because all sufficiently large integers can be written as sums of nonnegative seven cubes \cite{Siksek} and $1072$ can be written as sums of $r$ positive cubes with $2 \le r \le 8$. Thus, we can say
\[
	\frac{\#\mathcal{E}_k(N)}{N} \to 0 
	\quad\text{ as }\quad
	N \to \infty,
\]
where $\mathcal{E}_k(N)$ denotes the set of positive integers $\le N$ which cannot be written as sums of $k$ positive cubes for $k \geq 5$. Let $\Sigma^3_k$ be the set of sums of $k$ positive cubes. Then, since
\begin{align*}
	\frac{\#\bigl( p^r ([N]_p \cap \Sigma^3_k )\bigr)}{\#[p^r N]}
	&= \frac{\#\bigl( p^r ([N]_p \setminus \mathcal{E}_k(N))\bigr)}{\#[p^r N]} \\
	&= \frac{\#\bigl( [N]_p \setminus \mathcal{E}_k(N)\bigr)}{p^r N}
	\to \frac{1}{p^r}\bigl( 1 - \frac{1}{p} \bigr) > 0
\end{align*}
as $N \to \infty$, there exists $M$ such that $p \nmid M$, $M \in \Sigma^3_k$ and $p^r M \in \Sigma^3_k$. 

Thus, $f_k(p^r) = p^r$ by
\begin{align*}
	p^r M = f_k(p^r M) = f_k(p^r)\,f_k(M) = f_k(p^r) M.
\end{align*}

The proof is completed.

\section{Commutative conditions with sums of cubes}

As a variant problem, we may consider multiplicative functions $g_k$ satisfying
\[
	g_k(x_1^3+ x_2^3 + \cdots + x_k^3) = g_k(x_1)^3 + g_k(x_2)^3 + \cdots + g_k(x_k)^3
\]
for fixed $k \geq 2$. We can solve this problem in the similar way.

\begin{theorem}
If a multiplicative function $g_2$ satisfies
\[
	g_2(a^3+b^3) = g_2(a)^3 + g_2(b)^3
\]
for $a, b \in \mathbb{N}$, then $g_2(n) = n$ unless $3 \,|\, n$ and $3^2 \nmid n$. Besides, $g_2(3) = 3, \frac{3(-1+\sqrt{-3})}{2}$, or $\frac{3(-1-\sqrt{-3})}{2}$. 
\end{theorem}

\begin{proof}
We obtain
\[
	g_2(n) = n \text{ for }n = 1, 2, 4, 5, 7, 8, 9, 13, 19
\]
and $g_2(3)^3 = 27$ by using the following \texttt{Mathematica} code.

\begin{verbatim}
Solve[g1 == 1 && g2 == g1^3 + g1^3 && g9 == g2^3 + g1^3
  && g4 g7 == g3^3 + g1^3 && g5 g7 == g3^3 + g2^3
  && g5 g13 == g4^3 + g1^3 && g8 g9 == g4^3 + g2^3
  && g7 g13 == g4^3 + g3^3 && g2 g9 g7 == g5^3 + g1^3
  && g7 g19 == g5^3 + g2^3 && g8 g19 == g5^3 + g3^3,
{g1, g2, g3, g4, g5, g7, g8, g9, g13, g19}]
\end{verbatim}

Note that $g_2(27) = 27$ from $g_2(2)\,g_2(27) = g_2(3)^3+g_2(3)^3$. So, $g_2(3)^3 = g_2(3^3)$, but we cannot conclude $g_2(3) = 3$. We can show that $g_2$ fixes powers of $2$ and $3$ except for $g_2(3)$ by the same method as Lemma~\ref{lem:f2(2^r)}. But, we cannot use
\[
	g_2(9 \cdot 2^{3m}) = g_2(2^{3m} + 2^{3(m+1)})
	= g_2(2^m)^3 + g_2(2^{m+1})^3
\]
to show $g_2(2^{3(m+1)}) = 2^{3(m+1)}$. Instead, we use
\begin{align*}
	g_2(9 \cdot 2^{3(m+1)}) &= g_2(9 \cdot 8 \cdot 2^{3m}) 
	= g_2\bigl( (2^3 + 4^3)\,2^{3m} \bigr) \\
	&= g_2(2^{m+1})^3 + g_2(2^{m+2})^3.
\end{align*}

Next, $g_2$ fixes $p$ and $p^2-p+1$ or $\frac{p^2-p+1}{3}$ for prime $p \geq 5$ as in Lemmas~\ref{lem:f2(p)} and \ref{lem:f2(p^3),f2(p^2-p+1)}. Here, we cannot say whether $g_2$ fixes $p^3$ or not, because the condition yields $g_2(p)^3 = p^3$ not for $g_2(p^3)$.

It can be derived similarly that 
\[
	g_2(p^2+3) = p^2+3,\qquad
	g_2(p^2+1)^3 = (p^2+1)^3,\qquad
	g_2(p^4+3) = p^4+3
\]
as in Lemma~\ref{lem:f2_fixes_some_patterns}. Then, we can conclude that $g_2(p^r) = p^r$ for $r = 3m+1$ and $r = 3m+2$ by calculating inductively
\begin{align*}
	&g_2\!\left( p^{3m} \left( \frac{p+1}{2} \right)^3 + p^{3m} \left( \frac{p-1}{2} \right)^3 \right) = g_2(p^{3m+1})\,g_2\!\left( \frac{p^2+3}{4} \right), \\
	&g_2\!\left( p^{3m} \left( \frac{p^2-1}{2} \right)^3 + p^{3m} \left( \frac{p^2+1}{2} \right)^3 \right) = g_2(p^{3m+2})\,g_2\!\left( \frac{p^4+3}{4} \right).
\end{align*}

As for $g_2(p^{3m}) = p^{3m}$, we use induction on the equality
\begin{align*}
	g_2(p^{3m} + p^{3m}) 
	&= g_2(p^m)^3 + g_2(p^m)^3 \\
	&= g_2(p^{3m})\,g_2(2) = 2 g_2(p^{3m}).
\end{align*}
\end{proof}

\begin{theorem}
Fix $k \geq 3$. If a multiplicative function $g_k$ satisfies
\[
	g_k(a_1^3 + a_2^3 + \dots + a_k^3) = g_k(a_1)^3 + g_k(a_2)^3 + \dots + g_k(a_k)^3
\]
for all $a_1, a_2, \dots, a_k \in \mathbb{N}$, then $g_k$ is the identity function.
\end{theorem}

The proof is similar to that of Theorem \ref{thm:main}. But, we use $g_k(n)^3 = n^3$ instead of $g_k(n^3) = n^3$.

For $k=3$, we obtain $g_3(n) = n$ for $n = 1,2,3,4,5,9,11,43$ by the \texttt{Mathematica} code, if we set $h_n = g_3(n)$:
\begin{verbatim}
Solve[h1 == 1 && h3 == h1^3 + h1^3 + h1^3 
 && h2 h5 == h2^3 + h1^3 + h1^3 && h4 h9 == h3^3 + h2^3 + h1^3 
 && h43 == h3^3 + h2^3 + h2^3 && h5 h11 == h3^3 + h3^3 + h1^3 
 && h2 h3 h11 == h4^3 + h1^3 + h1^3 && h9 h11 == h4^3 + h3^3 + h2^3 
 && h3 h43 == h4^3 + h4^3 + h1^3,
{h1, h2, h3, h4, h5, h9, h11, h43}]     
\end{verbatim}

It can be shown similarly that $g_3$ fixes $2^{3m+1}$ and $2^{3m+2}$. But, we cannot use $g_3(5 \cdot 2^{3m+1})$ as in the proof of Lemma~\ref{lem:f3(2^r),f3(3^r)} to show that $2^{3(m+1)}$ is fixed. So, instead, we use
\begin{align*}
	g_3(3 \cdot 2^{3(m+1)}) 
	&= g_3\bigl( (2^3 + 2^3 + 2^3)\,2^{3m} \bigr) \\
	&= g_3(2^{m+1})^3 + g_3(2^{m+1})^3 + g_3(2^{m+1})^3.
\end{align*}
By the same reason, we use 
\[
	g_3(8 \cdot 3^{3(m+1)}) = g_3( 6^3 \cdot 3^{3m}) = g_3\bigl( (3^3+4^3+5^3)\,3^{3m} \bigr)
\]
to show $3^{m+1}$ is fixed. This method can be also applied to Lemma~\ref{lem:f3(2^r),f3(3^r)}.

Now, we can say that $g_3(n)^3 = n^3$ for all $n \in \mathbb{N}$ by the same method of Lemma~\ref{lem:f3(n^3)=n^3}. But, we are not yet sure $g_3(n^3) = n^3$. So, we are not sure that $\bigl(6 \lambda^2 (\lambda^3+3p^r)^2 \bigr)^3$ is fixed by $g_3$ in the LHS of Richmond's representation
\[
	g_3\bigl( p^r \bigl(6 \lambda^2 (\lambda^3+3p^r)^2 \bigr)^3 \bigr)
	= g_3(p^r)\,g_3\bigl( \bigl(6 \lambda^2 (\lambda^3+3p^r)^2 \bigr)^3 \bigr).
\]

However, $g_3$ fixes $(6N)^3$, since
\begin{align*}
	g_3\bigl( (6N)^3 \bigr) 
	&= g_3\bigl( (3N)^3 + (4N)^3 + (5N)^3 \bigr) \\
	&= g_3(3N)^3 + g_3(4N)^3 + g_3(5N)^3 \\
	&= (3N)^3 + (4N)^3 + (5N)^3 = (6N)^3.
\end{align*}

So, $g_3(p^r) = p^r$ for all prime $p$ and $r \geq 1$.

For $k=4$, we obtain $g_4(n) = n$ for $n = 1,2,3,4,5,7,9,11,37$, if we set $s_n = g_4(n)$:
\begin{verbatim}
Solve[s1 == 1 && s4 == s1^3 + s1^3 + s1^3 + s1^3
 && s11 == s2^3 + s1^3 + s1^3 + s1^3
 && s2 s9 == s2^3 + s2^3 + s1^3 + s1^3
 && s2 s3 s5 == s3^3 + s1^3 + s1^3 + s1^3
 && s37 == s3^3 + s2^3 + s1^3 + s1^3
 && s4 s11 == s3^3 + s2^3 + s2^3 + s1^3
 && s9 s7 == s3^3 + s3^3 + s2^3 + s1^3
 && s2 s5 s7 == s3^3 + s3^3 + s2^3 + s2^3
 && s2 s37 == s4^3 + s2^3 + s1^3 + s1^3,
{s1, s2, s3, s4, s5, s7, s9, s11, s37}]
\end{verbatim}

Then, we can show that $g_4$ fixes $2^{3m+1}$ and $2^{3m+2}$ by the similar ways as in Lemma~\ref{lem:f4(2^r)}. Also, $g_4(2^{3(m+1)}) = 2^{3(m+1)}$ can be derived by
\begin{align*}
	g_4(7 \cdot 2^{3(m+1)})
	&= g_4(7 \cdot 2^3 \cdot 2^{3m}) = g_4\bigl( (1^3+1^3+3^3+3^3)\,2^{3m} \bigr) \\
	&= g_4(2^{m})^3 + g_4(2^{m})^3 + g_4(3 \cdot 2^{m})^3 + g_4(3 \cdot 2^{m})^3.
\end{align*}

Thus, we can conclude that $g_4(p^r) = p^r$ by the similar way as in \S4.

The proof for $k \geq 5$ is almost identical to \S5. If we set $t_n = g_k(n)^3$, $t_n = 1$ or $t_n = n^3$ for $n \le 10$ as in Lemma~\ref{lem:fk(n^3)}. Thus, we can obtain $g_k(n)^3 = 1$ or $g_k(n)^3 = n^3$ for all $n$. But, the former case is impossible by $g_k(k) = k$.

\end{document}